\documentclass[letterpaper,12pt]{amsart}
\usepackage{graphicx}
\usepackage{amsthm, amsmath, amssymb}
\usepackage{fullpage}
\usepackage{mathrsfs,latexsym}
\usepackage[foot]{amsaddr}
\usepackage{enumitem}
\usepackage{verbatim}
\usepackage{varwidth}
\usepackage[b]{esvect}
\newtheorem{theorem}{Theorem}
\newtheorem{lem}[theorem]{Lemma}

\newtheorem{prop}[theorem]{Proposition}
\theoremstyle{remark}
\theoremstyle{definition}
\newtheorem{conj}[theorem]{Conjecture}
\numberwithin{equation}{section}
\newcommand{\abs}[1]{\left\lvert#1\right\rvert}
\newcommand{\set}[1]{\left\lbrace#1\right\rbrace}

\newcommand{\ab}{\allowbreak}

\DeclareMathOperator{\ex}{ex}

\DeclareMathOperator{\exlin}{\ex_{\rm local(in)}}

\DeclareMathOperator{\pil}{\pi_{\rm local}}
\DeclareMathOperator{\pilin}{\pi_{\rm local(in)}}
\DeclareMathOperator{\Gmax}{G_{\rm max}}
\DeclareMathOperator{\Gin}{G_{v({\rm in})}}
\DeclareMathOperator{\gin}{g_{v({\rm in})}}
\DeclareMathOperator{\Gmaxin}{G_{\max(in)}}
\DeclareMathOperator{\pilout}{\pi_{\rm local(out)}}
\DeclareMathOperator{\Gout}{G_{v({\rm out})}}
\DeclareMathOperator{\gout}{g_{v({\rm out})}}
\DeclareMathOperator{\Gmaxout}{G_{\max(out)}}
\DeclareMathOperator{\Gein}{G_{\emptyset(in)}}

\newcommand{\floor}[1]{\lfloor#1\rfloor}
\newcommand{\ds}{\displaystyle}

\begin{document}
%
\title{Maximum density of vertex-induced perfect cycles and paths in the hypercube}
\author{John Goldwasser$^*$}
\address{$^*$West Virginia University}
\email{jgoldwas@math.wvu.edu}
\author{Ryan Hansen$^*$}
\email{rhansen@math.wvu.edu}
%
\keywords{hypercube, inducibility, perfect cycle}

\begin{abstract}
Let $H$ and $K$ be subsets of the vertex set $V(Q_d)$ of the $d$-cube $Q_d$ (we call $H$ and $K$ configurations in $Q_d$).  We say $K$ is an \emph{exact copy} of $H$ if there is an automorphism of $Q_d$ which sends $H$ to $K$. If $d$ is a positive integer and $H$ is a configuration in $Q_d$, we define $\pi(H,d)$ to be the limit as $n$ goes to infinity of the maximum fraction, over all subsets $S$ of $V(Q_n)$, of sub-$d$-cubes of $Q_n$ whose intersection with $S$ is an exact copy of $H$.
  We determine $\pi(C_8,4)$ and $\pi(P_4,3)$ where $C_8$ is a ``perfect'' 8-cycle in $Q_4$ and $P_4$ is a ``perfect'' path with 4 vertices in $Q_3$, and make conjectures about $\pi(C_{2d},d)$ and $\pi(P_{d+1},d)$ for larger values of $d$.  In our proofs there are connections with counting the number of sequences with certain properties and with the inducibility of certain small graphs.  In particular, we needed to determine the inducibility of two vertex disjoint edges in the family of bipartite graphs.
\end{abstract}

\maketitle

\section{Background} 
\label{sec:background}
	The $n$-cube, which we denote by $Q_n$, is the graph whose vertex set is the set of all binary $n$-tuples, with two vertices adjacent if and only if they differ in precisely one coordinate (so Hamming distance 1). Let $[n]=\set{1,2,\ldots,n}$.  We sometimes denote a vertex $\left[x_1,x_2,\ldots,x_n\right]$ of $Q_n$ by the subset $S$ of $[n]$ such that $i\in S$ if and only if $x_i=1$. So if $n=4$, then $\emptyset$ denotes $\left[ 0 0 0 0\right]$, and $\set{1,3}$ or 13, denotes $\left[ 1 0 1 0\right]$ and $\set{\set{1},\set{1,3}}$ (or $\set{1,13}$) denotes $\set{\left[ 1 0 0 0 \right], \left[ 1 0 1 0 \right]}$.  The weight of a vertex is the number of 1s.  For each positive integer $d$ less than or equal to $n$, $Q_n$ has $\binom{n}{d}2^{n-d}$ subgraphs which are isomorphic to $Q_d$ ($d$ coordinates can vary, while $n-d$ coordinates are fixed).
	
	Let $H$ and $K$ be subsets of $V(Q_d)$ (we call $H$ and $K$ configurations in $Q_d$).  We say $K$ is an \emph{exact copy} of $H$ if there is an automorphism of $Q_d$ which sends $H$ to $K$.  For example, $\set{\emptyset,12}$ is an exact copy of $\set{2,123}$ in $Q_3$, but $\set{2,13}$ is not (the vertices are distance 3 apart). So if $K$ is an exact copy of $H$ then they induce isomorphic subgraphs of $Q_d$, but the converse may not hold.
	
	Let $d$ and $n$ be positive integers with $d\leq n$, let $H$ be a configuration in $Q_d$ and let $S$ be a subset of $V(Q_n)$.  We let $G(H,d,n,S)$ denote the number of sub-$d$-cubes $R$ of $Q_n$ in which $S\cap R$ is an exact copy of $H$, $\Gmax(H,d,n)=\max_{S\subseteq V(Q_n)} G(H,d,n,S)$, $g(H,d,n,S)=\frac{G(H,d,n,S)}{\binom{n}{d}2^{n-d}}$ denote the fraction of sub-$d$-cubes $R$ of $Q_n$ in which $S\cap R$ is an exact copy of $H$, and $\ex(H,d,n)=\frac{\Gmax(H,d,n)}{\binom{n}{d}2^{n-d}}=\max_{S\subseteq V(Q_n)} g(H,d,n,S)$.  Note that $\ex(H,d,n)$ is the average of $2n$ densities $g(H,d,n-1,S_j)$, each of them the fraction of sub-$d$-cubes $R$ in a sub-$(n-1)$-cube of $Q_n$ in which $R\cap S_j$ is an exact copy of $H$, where $S_j$ is the intersection of a maximizing subset $S$ of $V(Q_n)$ with one of the $2n$ sub-$(n-1)$-cubes.  Hence $\ex(H,d,n)$ is the average of $2n$ densities, each of them less than or equal to $\ex(H,d,n-1)$, which means $\ex(H,d,n)$ is a nonincreasing function of $n$, so we can define the $d$-cube density $\pi(H,d)$ of $H$ by
	\[
		\pi(H,d)=\lim_{n\to\infty} ex(H,d,n).
	\]
So $\pi(H,d)$ is the limit as $n$ goes to infinity of the maximum fraction, over all $S\subseteq V(Q_n)$, of ``good'' sub-$d$-cubes -- those whose intersection with $S$ is an exact copy of $H$.

As far as we know, this paper is the first to define the notion of $d$-cube density.  There have been many papers on Tur\'{a}n and Ramsey type problems in the hypercube.  There has been extensive research on the maximum fraction of edges of $Q_n$ one can take with no cycle of various lengths \cite{Chung:1992fk,Conlon:2010,Furedi:2015,Thomason:2009} and a few papers on vertex Tur\'{a}n problems in $Q_n$ \cite{Johnson:1989cy,kostochka1976piercing}.  There has also been extensive work on which monochromatic cycles must appear in any edge-coloring of a large hypercube with a fixed number of colors \cite{Alon:2006,Axenovich:2006,Chung:1992,Conder:1993}, and a few results on which vertex structures must appear \cite{Goldwasser:2012}.

In \cite{AKS:2006,Goldwasser:2018,Offner:2008} results were obtained on the polychromatic number of $Q_d$ in $Q_n$, the maximum number of colors in an edge coloring of a large $Q_n$ such that every sub-$d$-cube gets all colors.

We wanted to investigate a different extremal problem in the hypercube: the maximum density of a small structure within a subgraph of a large hypercube.  Instead of using graph isomorphism to determine if two substructures are the same, it seemed to capture the essesnce of a hypecube better if the small structure was ``rigid'' within a sub-$d$-cube, and that is what motivated our definition of $d$-cube density.

There are strong connections between $d$-cube density and \emph{inducibility} of a graph, a notion of extensive study over the past few years.  Given graphs $G$ and $H$, with $\abs{V(G)}=n$ and $\abs{V(H)}=k$, the \emph{density} of $H$ in $G$, denoted $d_H(G)$, is defined by
	\[
		d_H(G)=\frac{\textrm{\# of induced copies of $H$ in $G$} }{\binom{n}{k}}
	\]
Pippenger and Golumbic \cite{PG:1975} defined the \emph{inducibility} $I(H)$ of $H$ by
	\[
		I(H)=\lim_{n\to\infty}\max_{\abs{V(G)}=n}d_H(G).
	\]
	Within the past few years $I(H)$ has been determined for all graphs $H$ with 4 vertices except the path $P_4$ \cite{exoo1986dense,Hirst:2014jp}\cite{EvenZoharLinial:2015}.
	
	Given a graph $H$, a natural candidate for maximizing the number of induced copies of $H$ is a balanced blow-up of $H$.  Equipartition the $n$ vertices into $\abs{V(H)}=k$ classes corresponding to the verteices of $H$ and add all possible edges between each pair of parts corresponding to an edge of $H$.  Any $k$-subset which has one vertex in each part will induce a copy of $H$, so $I(H)\geq \frac{k!}{k^k}$ for any graph $H$ with $k$ vertices.  Iterating blow-ups of $H$ within each part improves the bound to $I(H)\geq \frac{k!}{k^k-k}$.
	
	A natural generalization of $I(H)$ is to restrict $G$ to a particular class of graphs.  Let $\mathscr{G}$ be a class of graphs.  The inducibility of $H$ in $\mathscr{G}$ is defined by
	\[
		I(H,\mathscr{G})=\lim_{n\to\infty}\max_{\abs{V(G)}=n,G\in\mathscr{G}}d_H(G),
	\]
	if the limit exists (if $\mathscr{G}$ is all graphs the limit always exists).  Let $\mathscr{T}$ be the family of all triangle-free graphs.  Hatami et al. \cite{Hatami:2013} and Grzesik \cite{Grzesik:2012} used flag algebras to show that $I(C_5,\mathscr{T})=\frac{5!}{5^5}=\frac{24}{625}$, achieving the non-iterated blow-up lower bound.  
	
	In \cite{CLP:2020}, Choi, Lidicky, and Pfender consider the inducibility of oriented graphs (directed graphs with no 2-cycles).  For the directed path $\vv{P_k}$ they conjectured that
	\[
		I(\vv{P_k})=\frac{k!}{(k+1)^{k-1}-1}
	\]
	the lower bound provided by an iterated blow-up of the directed cycle $\vv{C_{k+1}}$.  To eliminate the possibility of iterated blow-ups they considered the family $\vv{\mathcal{T}}$ of oriented graphs with no transitive tournament on three vertices (so every 3-cycle is directed).  They conjectured that
	\[
		I(\vv{P_k},\vv{\mathcal{T}})=\frac{k!}{(k+1)^{k-1}}
	\]
	Again, the lower bound is provided by a blow-up of $\vv{C_{k+1}}$ (no iterations).  They used flag algebras to prove their conjecture for $k=4$:
	\[
		I(\vv{P_4},\vv{\mathcal{T}})=\frac{4!}{5^3}=\frac{24}{125}
	\]
	
	It has been shown \cite{Bollobas:1986bfa,BrownSidorenko:1994} that if $H$ is a complete bipartite graph then the graph that maximizes $I(H)$ can be chosen to be complete bipartite.  There are also a few results on inducibility of 3-graphs \cite{FalgasRavry:2012uu}.
	
	As with inducibility, $d$-cube density is exceedingly difficult to determine for all but a few configurations $H$ and values of $d$.  A different kind of blow-up can be used to produce lower bounds.  For a few configurations $H$ we have been able to prove a matching upper bound, generally using results on inducibility to do so.  In this paper we determine the $d$-cube density of a ``perfect'' path with 4 vertices in $Q_3$ and a ``perfect'' 8-cycle in $Q_4$.

\section{Results} 
\label{sec:results}

	We have some results in a forthcoming paper for certain configurations $H$ when $d$ is equal to 2, 3, or 4, and for a couple of infinite families with $d$ any integer greater than 2.  If $H$ is two opposite vertices in $Q_2$, clearly $\pi(H,2)=1$ (let $S$ be all vertices in $Q_n$ of even weight).  A more interesting example is when $H$ is two adjacent vertices in $Q_2$.  Then it is not hard to show that $\pi(H,2)=\frac{1}{2}$.  (For the lower bound, take $S$ to be all vertices in $Q_n$ such that the sum of coordinates 1 through $\floor{\frac{n}{2}}$ is even.  Any sub-2-cube which has one varying coordinate in and one out of $\left[1,\floor{\frac{n}{2}}\right]$ will have an exact copy of $H$.)  We have been unable to determine $\pi(W_d,d)$ for any $d\geq 2$ when $W_d$ is a single vertex in $Q_d$.  Letting $S$ be the set of all vertices in $Q_n$ with weight a multiple of 3 shows that $\pi(W_2,2)\geq \frac{2}{3}$.  Using flag algebras Rahil Baber \cite{Baber:2014p} has shown that $\pi(W_2,2)\leq .686$.  We suspect that for sufficiently large $d$ one cannot do better than choosing vertices randomly with uniform probability $\frac{1}{2}^d$, wwhich gives a density of $\frac{1}{e}$ in the limit as $d$ goes to infinity.
	
	Let $P_{d+1}$ denote the vertex set of a path in $Q_d$ with $d+1$ vertices whose endpoints are Hamming distance $d$ apart.  We call $P_{d+1}$ a \emph{perfect path}.  For example, $\set{\emptyset,1,12,123,1234}$ and $\set{13,3,\emptyset,4,24}$ are both perfect paths in $Q_4$, while $\set{13,3,\emptyset,4,14}$ is not, even though these 5 vertices do induce a graph-theoretic path.
	
	Let $C_{2d}$ denote the vertex set of a $2d$-cycle in $Q_d$ where all $d$ opposite pairs of vertices are distance $d$ apart.  We call $C_{2d}$ a \emph{perfect $2d$-cycle}.  The only graph-theoretic induced 6 cycle in $Q_3$ is perfect, but while $\set{\emptyset,1,12,123,1234,234,34,4}$ is a perfect 8-cycle, $\set{\emptyset,1,12,123,23,234,34,4}$ and $\set{\emptyset,1,12,123,1234,134,34,3}$ induce nonisomorphic 8-cycles in $Q_4$ which are not perfect.
	
	The main results in this paper are the two following theorems.
	
	\begin{theorem}\label{PC8theorem}
		$\pi(C_8,4)=\frac{3}{32}$
	\end{theorem}
	
	\begin{theorem}\label{P4theorem}
		$\pi(P_4,3)=\frac{3}{8}$
	\end{theorem}
	
	These are special cases of the following conjectures.
	
	\begin{conj}\label{cycleconj}
		$\pi(C_{2d},d)=\frac{d!}{d^d}$ for all $d\geq 4$.
	\end{conj}
	
	\begin{conj}\label{pathconj}
		$\pi(P_{d+1},d)=\frac{d!}{(d+1)^{d-1}}$ for all $d\geq 3$.
	\end{conj}
	
	Note that the formulas in these two conjectures are the same as in the conjectures about the inducibility of directed cycles and paths in oriented graphs.  Conjecture \ref{cycleconj} is significant because, as we show in Proposition \ref{mindensity}, $\pi(H,d)\geq \frac{d!}{d^d}$ for all configurations $H$ in $Q_d$ for all $d\geq 1$.  To show $\frac{d!}{d^d}$ is also an upper bound when $d=4$ we needed to find the inducibility of two vertex disjoint edges in the family of all bipartite graphs.  To prove both Theorem \ref{PC8theorem} and Theorem \ref{P4theorem} we show that the $d$-cube density we are trying to determine is equal to the fraction of $d$-sequences of an $n$-set which have certain properties and then we solve the sequence problems.
	
	


\section{Constructions} 
\label{sec:constructions}

	Consider the following construction which gives a lower bound for the $d$-cube density of any configuration $H$ in $Q_d$, for any $d$.  Let $[n]$ denote the set $\{1,2,\ldots,n\}$.  We partition $[n]$ into $A_1,A_2,\ldots,A_d$ and let $B$ be the set of binary $d$-tuples representing $H$.  For each vertex $\vec{v}=\left[v_1,v_2,\ldots,v_n\right]$ in $Q_n$ we let $\vec{v}(A_i)$ equal 0 or 1 according to $\ds\vec{v}(A_i)\equiv \sum_{j\in A_i}v_j \mod 2$.  We put $\vec{v}$ in $S$ if and only if the $d$-tuple $\left( \vec{v}(A_j) \right)_{j\in[d]}$ is in $B$.  For example, for a perfect 8-cycle in $Q_4$, we could have $B=\{0000,\ab 1000,\ab 1100,\ab 1110,\ab 1111,\ab 0111,\ab 0011,\ab 0001\}$ and $\vec{v}$ would be in $S$ if and only if its number of 1s in coordinates in $A_1,A_2,A_3,A_4$ is either even,even,even,even, or odd,even,even,even, and so on.  We observe that if a sub-$d$-cube has one coordinate in each of $A_1,A_2,\ldots,A_d$, then it will contain an exact copy of $H$. By taking an equipartition of $[n]$, we find the following lower bound:
	
	\begin{prop}\label{mindensity}
		$\ds \pi(H,d)\geq \frac{d!}{d^d}$ for all configurations $H$ in $Q_d$ for all positive integers $d$.
	\end{prop}
	
	We call a set $S$ constructed in this way a \emph{blow-up of $H$}.  This notion of blow-up is clearly related to, but not the same as, the blow-up of a graph $G$ (for one thing a blow-up of a graph has one part for each vertex, whereas a blow up of a configuration in $Q_d$ has $d$ parts).  In $Q_2$, the only configuration $H$ for which equality holds in Proposition \ref{mindensity} is two adjacent vertices.  The smallest upper bound for any of the 22 possible configurations in $Q_3$ as computed by Rahil Baber using flag algebras is $.3048$ (when $H$ is two adjacent vertices in $Q_3$), so it is highly unlikely that any configuration in $Q_3$ has 3-cube density equal to $\frac{2}{9}$, the lower bound provided by Proposition \ref{mindensity}.  Of the 238 possible configurations in $Q_4$, only three have flag algebra calculated upper bound 4-cube densities less than $.1$: one is the perfect 8-cycle, for which Theorem \ref{PC8theorem} says the exact value is $\frac{3}{32}=.09375$ and another is a graph theoretic, but not perfect, induced 8-cycle, with flag algebra 4-cube density upper bound $.094205$.  So there seems to be something special about the perfect 8-cycle.
	
	For the perfect path $P_{d+1}$ in $Q_d$ it turns out that a blow-up of $C_{2d+2}$ gives a better lower bound than that provided by Proposition \ref{mindensity}:
	
	\begin{prop}\label{path_density}
		$\ds\pi(P_{d+1},d)\geq\frac{d!}{(d+1)^{d-1}}$ for all positive integers $d$.
		
		\begin{proof}
			Let $S$ be a blow up of $C_{2d+2}$.  That is we partition $[n]$ into $A_1,\ab A_2,\ldots,\ab A_{d+1}$ and let $B$ be the set of binary $(d+1)$-tuples in a copy of $C_{2d+2}$.
			
			For each vertex $\vec{v}=[v_1,\ldots,v_n]$ in $Q_n$, we let $\vec{v}(A_i)$ equal 1 or 0 according to $\ds\vec{v}(A_i)\equiv \sum_{j\in A_i}v_j \mod 2$.  We put $\vec{v}$ in $S$ if and only if the $(d+1)$-tuple $\left( \vec{v}(A_j) \right)_{j\in[d+1]}$ is in $B$.  If a sub-$d$-cube has one coordinate in each of $d$ parts (and none in the other), then it will contain an exact copy of $P_{d+1}$.  For example, if $d=3$ and $B=\{0000,\ab 1000,\ab 1100,\ab 1110,\ab 1111,\ab 0111,\ab 0011,\ab 0001\}$ and we select a sub-3-cube with one coordinate in each of $A_1,A_2,$ and $A_4$(so each coordinate in $A_3$ is fixed) then if $\vec{v}(A_3)=0$ the 4-tuples $0001, \ab 0000,\ab 1000,\ab 1100$ in $B$ give us an exact copy of $P_4$ in any such sub-3-cube, while if $\vec{v}(A_3)=1$, then $1110,\ab 1111,\ab 0111,\ab 0011$ does the same.  If it is an equipartition, selecting the coordinates of the sub-$d$-cube one-by-one shows that
			\[
				\pi(P_{d+1},d)\geq \frac{(d+1)!}{(d+1)^{d}} = \frac{d!}{(d+1)^{d-1}}.
			\]
		\end{proof}
	\end{prop}
	

\section{Local density, perfect cycles, and sequences} 
\label{sec:local_density_perfect_cycles_and_sequences}
	
	Let $H$ be a configuration in $Q_d$ and $S$ be a subset of $V(Q_n)$.  For each vertex $v$ in $S$, we let $\Gin(H,d,n,S)$ be the number of sub-$d$-cubes $R$ of $Q_n$ containing $v$ in which $S\cap R$ is an exact copy of $H$, $\Gmaxin(H,d,n)=\max_{v\in S} \Gin(H,d,n,S)$ where the max is over all $v$ and $S$ such that $v\in S$, $\gin(H,d,n,S)=\frac{\Gin(H,d,n,S)}{\binom{n}{d}}$ denote the fraction of sub-$d$-cubes $R$ of $Q_n$ containing $v$ in which $S\cap R$ is an exact copy of $H$, and $\exlin(H,d,n)=\frac{\Gmaxin(H,d,n)}{\binom{n}{d}}$.  As with $\ex(H,d,n)$, a simple averaging argument shows that $\exlin(H,d,n)$ is a nonincreasing function of $n$, so we define $\pilin(H,d)$ by
	\[
		\pilin(H,d)=\lim_{n \to \infty}\exlin(H,d,n)
	\]
	For each vertex $v\not\in H$ we perform a similar procedure to define $\Gout(H,d,n,S)$, $\Gmaxout(H,d,n)$, $\gout(H,d,n,s)$, and $\pilout(H,d)$.  So $\pilin(H,d)$ and $\pilout(H,d)$ are the maximum local densities of sub-$d$-cubes with an exact copy of $H$ among all sub-$d$-cubes containing $v$ in $S$ and out of $S$ respectively.  Finally, we define $\pil(H,d)$ to be $\max\{\pilin(H,d),\pilout(H,d)\}$.  Since the global density cannot be more than the maximum local density, we must have $\pi(H,d)\leq \pil(H,d)$.

	
	For most configurations $H$ for which we have been able to determine $\pi(H,d)$, our procedure has been to prove an upper bound for $\pil(H,d)$ which matches the density of a construction.
	
	If $S\subseteq V(Q_n)$ we let $\overline{S}$ denote $V(Q_n)\setminus S$ and if $H$ is a configuration in $Q_d$ we let $\overline{H}$ denote $V(Q_d)\setminus H$.  Clearly $\pi(H,d)=\pi(\overline{H},d)$ and $\pilin(H,d)=\pilout(\overline{H},d)$.  If $H$ is self-complementary in $Q_d$, i.e. $\overline{H}$ is an exact copy of $H$, then $\pilin(H,d)=\pilout(\overline{H},d)=\pilout(H,d)=\pil(H,d)$.  Every configuration in $Q_3$ with 4 vertices is self-complementary, including $P_4$, and it is easy to check that $C_8$ is self-complementary in $Q_4$ (the complements of the two non-perfect induced 8-cycles in $Q_4$ are not 8-cycles).
	
	Now we pose and solve a different maximization problem whose answer we will show to be $\pi(C_8,4)$.  Let $S$ be a set of size $n$ and $d$ a positive integer.  We consider a set $A(d,n)$ of sequences of $d$ distinct elements of $S$.  Given a sequence $w$ in $A(d,n)$ an \emph{end-segment of w} is the set of the first $j$ elements of $w$ or the set of the last $j$ elements of $w$, for some $j$ in $[1,d)$.  We say the set $A(d,n)$ has Property $U$ if the two following conditions are satisfied:
	\begin{enumerate}
		\item For each pair of sequences $w$ and $x$ in $A(d,n)$, if $D$ is an end-segment of $w$ and all elements of $D$ are in the sequence $x$, then $D$ is an end-segment of $x$ with elements in the same order as in $w$ (so if $abc$ is the beginning of $w$, and $a,b,$ and $c$ all appear in $x$, then either $x$ begins $abc$ or ends $cba$).
		\item A sequence and its reversal are not both in $A(d,n)$ (unless $d=1$).
	\end{enumerate}
	
	For example, if $x$ and $w$ are sequences in a set $A(5,n)$ with Property $U$ and if $x$ is $abcde$, then $w$ cannot be $abceg$ (or its reversal), $abegh$ (or its reversal), or $ghiaj$ (or its reversal), but could be $fbdcg$ (or its reversal) or $edgbh$ (or its reversal).  It is easy to see that no two sequences in $A(d,n)$ can have the same set of $d$ elements.
	
	Let $T(d,n)$ denote the maximum size of a family $A(d,n)$ with Property $U$.
	
	\begin{prop}
		$\Gmaxin(C_{2d},d,n)=T(d,n)$ for all $d\geq 2$.
		\begin{proof}
			Without loss of generality, we can assume that $\emptyset$ is a vertex where the local $d$-cube density of $C_{2d}$ is a maximum and that $\Gein(C_{2d},d,n,S)=\Gmaxin(C_{2d},d,n)$.  Now we construct a set $A(d,n)$ of $d$-sequences.
			
			The sequence $a_1,a_2,\ldots,a_d$ or its reversal is in $A(d,n)$ if and only if the intersection of $S$ and the sub-$d$-cube where $a_1,a_2,\ldots,a_d$ are the nonconstant coordinates is equal to $\{\emptyset,\ab a_1,\ab a_1 a_2,\ab a_1 a_2 a_3,\ldots,\ab a_1 a_2\cdots a_k,\ab a_2 a_3\cdots a_d,\ldots,\ab a_{d-1}a_d,\ab a_d\}$.  We claim that $A(d,n)$ has Property $U$.
			
			Suppose it does not, say $x$ and $w$ are sequences in $A(d,n)$ with $a_1 a_2 \cdots a_j$ an end-segment in $x$ all of whose elements are in $w$ but not an end-segment in $w=b_1 b_2\cdots b_k$.  Then $\{a_1,\ab a_2,\ldots,\ab a_j\}$ is a subset of $\{b_1,\ab b_2,\ldots,\ab b_k\}$, so is another vertex in $S$ which is in the sub-$d$-cube containing the perfect $2d$-cycle $\{\emptyset,\ab b_1,\ab b_1 b_2,\ldots,\ab b_{d-1}b_d,\ab b_d\}$, a contradiction.
			
			Similarly, by reversing the procedure, a family of sequences with Property $U$ and size $T(d,n)$ can be used to construct $T(d,n)$ sub-$d$-cubes containing $\emptyset$ with exact copies of $C_{2d}$.
		\end{proof}
	\end{prop}
	
	We define $t(d,n)$ to be $\frac{T(d,n)}{\binom{n}{d}}$.  Hence $t(d,n)=\frac{\Gmaxin(C_{2d},d,n)}{\binom{n}{d}}=\exlin(C_{2d},d,n)$ is a nonincreasing function of $n$, so we can define $t(d)$ by $t(d)=\lim_{n \to \infty}t(d,n)=\lim_{n \to \infty}\exlin(C_{2d},d,n)=\pilin(C_{2d},d)$.  Hence we have
	
		\begin{prop}\label{pilin=td}
			For all $d\geq 2$
			\[
				\pilin(C_{2d},d)=t(d)
			\]
		\end{prop}
		
		We now calculate $t(3)$.
	
		Let $A(3,n)$ be a set of 3-sequences with Property $U$.  No symbol can appear at the end in one sequence and in the middle of another, so we let $A$ be the set of symbols which appear at the beginning or end and $B$ be the set of symbols which appear in the middle.  If $\abs{A}=m$ and $\abs{B}=p$, then, since a sequence and its reversal cannot both be in $A(3,n)$, the total number of sequences is at most $\binom{m}{2}p=\frac{(n-m)m(m-1)}{2}$ which is maximized when $m=\left\lceil{\frac{2n}{3}}\right\rceil$.  Hence,
		\[
			\pilin(C_6,3) = t(3) = \lim_{n\to\infty} \frac{\left( \frac{2}{3}n \right)^2 \left( \frac{1}{3}n \right)}{2\binom{n}{3}}=\frac{4}{9}.
		\]
		
		To find $\pilout(C_6,3)$, we just note that if $S$ is the set of all vertices in $Q_n$ with weight not divisible by 3, then every $Q_3$ containing $\vv{0}$ has an exact copy of $C_6$ (the unique vertex with weight 3 in each $Q_3$ containing $\phi$ is also not in $C_6$), so $\pilout(C_6,3)=1$.  Using this same set $S$, it is not hard to show that $\pi(C_6,3)\geq \frac{1}{3}$ (any $Q_3$ whose smallest weight vertex is a multiple of 3 has an exact copy of $C_6$).  We have been unable to show equality, but Baber's flag algebra upper bound of $.3333333336$ would seem to show equality must hold.
	
	
	To prove Theorem \ref{PC8theorem}, we will prove a result about inducibility in bipartite graphs.
	
	\begin{theorem}\label{bipartite_max_4_set_matching}
		Let $G$ be a bipartite graph with $n$ vertices.  Then the limit as $n$ goes to infinity of the maximum fraction of sets of 4 vertices of $G$ which induce two disjoint edges is equal to $\frac{3}{32}$.
		\begin{proof}
			Suppose $M,P$ is a bipartition of $V(G)$ where $\abs{M}=m$ and $\abs{P}=p$.  Let $\set{u_1,u_2,\ldots,u_m}$ and $\set{v_1,v_2,\ldots,v_p}$ be the vertices of $M$ and $P$ with respective degrees $r_1,r_2,\ldots,r_m$ and $c_1,c_2,\ldots,c_p$.  For $i\neq j$, let $t_{i,j}$ denote the number of vertices in $P$ which are adjacent to both $u_i$ and $u_j$.  Hence the total number of ``good'' sets of 4 vertices is
			\[
				N=\sum_{i<j} (r_i-t_{i,j})(r_j-t_{i,j})
			\]
			where the sum is over all pairs $i,j$ such that $1\leq i<j\leq m$.  To get an upper bound for this we first get an upper bound on the sum $S$ of all pairs of the factors in the products:
			\[
				S = \sum_{i<j}\left[ (r_i-t_{i,j}) + (r_j-t_{i,j})\right] = (m-1)\sum_{i=1}^m r_i - 2\sum_{i=1}^p \binom{c_j}{2}.
			\]
			This is because each $r_i$ appears in a sum with each $r_j$ where $j\neq i$ and because $\sum_{i=1}^m t_{i,j}=\binom{c_j}{2}$ since each pair of edges adjacent to $v_j$ is counted precisely once in the sum.

			Let $w=\sum_{i=1}^m r_i = \sum_{j=1}^p c_j$.  Then
			\begin{align*}
				S	&= (m-1)\sum_{i=1}^m r_i - \sum_{j=1}^p c_j^2 + \sum_{j=1}^p c_j\\
					&= mw-\sum_{j=1}^p c_j^2\\
					&\leq mw-\sum_{j=1}^p\left( \frac{w}{p} \right)^2\\
					&= mw - \frac{w^2}{p}
			\end{align*}
			where the inequality is by Cauchy-Schwartz.

			The function $f(w)=mw-\frac{w^2}{p}$ is maximized when $w=\frac{mp}{2}$, so $S\leq \frac{m^2p}{4}$.

			Now, we return to our consideration of $N=\sum_{i<j} (r_i-t_{i,j})(r_j-t_{i,j})$.

			The product $(r_i-t_{i,j})(r_j-t_{i,j})$ is at most $\left(\frac{p}{2}\right)^2$, achieved when $r_i=r_j=\frac{p}{2}$ and $t_{i,j}=0$, in which case $(r_i-t_{i,j})+(r_j-t_{i,j})=p$.  Since $S\leq\frac{m^2p}{4}$, to maximize $N$ we want to have $\frac{m^2}{4}$ products of $\left( \frac{p}{2} \right)^2$, with all other products being $0\cdot 0$.  Hence $N\leq \frac{m^2p^2}{16}$.  Since $m+p=n$, this is maximized when $m=p=\frac{n}{2}$, so $N\leq\frac{n^4}{256}$.  It is not hard to conclude from the above inequalities that equality holds only when $G$ is two disjoint copies of $K_{\frac{n}{4},\frac{n}{4}}$. This is the same graph which maximizes, among all graphs with $n$ vertices, the number of induced subgraphs with 4 vertices consisting of two edges which share a vertex and an isolated vertex.

			This gives a fraction of ``good'' sets of 4 vertices as 
			\[
				\frac{\frac{n^4}{256}}{\binom{n}{4}} = \frac{n^3}{(n-1)(n-2)(n-3)}\cdot \frac{3}{32}.
			\]
		\end{proof}
	\end{theorem}

	\begin{proof}[Proof of Theorem \ref{PC8theorem}]
		By Proposition \ref{mindensity}, $\pi(C_8,4)\geq\frac{3}{32}$.  Since $C_8$ is self-complementary in $Q_4$, $\pi(C_8,4)\leq\pil(C_8,4)=\pilin(C_8,4)=t(4)$ the last equality by Proposition \ref{pilin=td}.  So to complete the proof we just need to show that $t(4)\leq \frac{3}{32}$.

	Let $A(4,n)$ be a maximum size set of 4-sequences with Property $U$ with elements from an $n$-set.  Let $A=\{i,\in[n]:i \textrm{ is the first or last element in a sequence in }A(4,n)\}$ and let $B=[n]\setminus A$.  We construct a bipartite graph $G$ with vertex bipartition $A,B$.  If $a\in A$ and $b\in B$, then $[a,b]$ is an edge of $G$ if and only if $a$ and $b$ are consecutive elements in some sequence in $A(4,n)$ (so some sequence begins $ab$ or ends $ba$).
	
	Suppose $a_1 b_1 b_2 a_2$ is a sequence in $A(4,n)$.  Then $[a_1,b_1]$ and $[a_2,b_2]$ are edges of $G$.  Suppose $a_1 b_2$ (or $a_2 b_1$) is also an edge.  Then $a_1 b_2$ is an end-segment of some sequence in $A(4,n)$, which is impossible because $\{ a_1, b_2\}\subseteq \{ a_1, \ab b_1,\ab b_2,\ab a_2 \}$, but $a_1 b_2$ is not an end-segment in $a_1 b_1 b_2 a_2$.  Hence the size $T(4,n)$ of $A(4,n)$ is at most the number of sets of 4 vertices in $G$ which induce two disjoint edges, and by Theorem \ref{bipartite_max_4_set_matching}, 
	\[
		t(4)=\lim_{n\to\infty}\frac{T(4,n)}{\binom{n}{4}}\leq \frac{3}{32}.
	\]
	\end{proof}

	So Conjecture~\ref{cycleconj} is true for $d=4$.  Zongchen Chen \cite{Chen} has shown that $t(d)=\frac{d!}{d^d}$ for all $d\geq 4$.  While we know that $\pilin(C_{2d},d)=t(d)=\frac{d!}{d^d}$ for all $d\geq 4$, we have been unable to show that $\pilout(C_{2d},d)=\pilin(C_{2d},d)$ if $d\geq 5$ (our proof for $d=4$ used the fact that $C_8$ is self-complementary in $Q_4$), which would complete a proof of Conjecture 3.

	We have seen that equality in Conjecture \ref{cycleconj} does not hold for $d=3$ (since $\pi(C_6,3)\geq \frac{1}{3}$).


\section{Perfect Paths} 
\label{sec:perfect_paths}

To determine the $d$-cube density of $P_4$ in $Q_3$, as mentioned in Section \ref{sec:local_density_perfect_cycles_and_sequences}, our procedure is to prove an upper bound for $\pil(P_4,3)$ which matches the density of the construction given in Proposition \ref{path_density}.  We will show

\begin{prop}\label{pilP_4,3}
	$\ds\pil(P_4,3)\leq \frac{3}{8}$.
\end{prop}

Let $S$ be a set of size $n$ and $d$ a positive integer.  Let $R(d,n)$ be a family of pairs of sequences of elements in $S$, which we call \emph{$d$-bisequences}, one sequence of length $i$ and the other of length $d-i$, where $i\in [0,d]$, and where the $d$ elements in the pair of sequences are distinct.  Given a bisequence $w=\{a_1;a_2\}$ in $R(d,n)$ an \emph{initial-segment of $w$} is the set of the first $j$ elements of either $a_1$ or $a_2$ where $j\in [1,d]$.  We say that the set $R(d,n)$ has Property $V$ if the following conditions are satisfied:
\begin{enumerate}
	\item For each pair of bisequences $w=\{a_1;a_2\}$ and $x=\{a_3;a_4\}$ in $R(d,n)$, if $L$ is an initial segment of $w$ and all elements of $L$ are in the $d$-bisequence $x$, then $L$ is an initial segment of one of the sequences in $x$ with elements in the same order as in $w$.
	\item A bisequence $\{a_1;a_2\}$ and its reversal $\{a_2;a_1\}$ are not both in $R(d,n)$.
\end{enumerate} 

Let $B(d,n)$ denote the maximum size of a family of bisequences $R(d,n)$ with property $V$ and let $b(d,n)=\frac{B(d,n)}{\binom{n}{d}}$.

\begin{prop}\label{path-bisequence}
	$\exlin(P_{d+1},d,n)=b(d,n)$.
	\begin{proof}
		Without loss of generality, we can assume that $\emptyset$ is a vertex where the local $d$-cube density of $P_{d+1}$ is a maximum.  Now we construct a set $R(d,n)$ of bisequences.
		
		The bisequence $\{(a_1,a_2,\ldots,a_j);(b_1,b_2,\ldots,b_i)\}$ or its reversal is in $R(d,n)$ if and only if the intersection of $S$ and the sub-$d$-cube where $a_1,a_2,\ldots,a_j,b_1,b_2,\ldots,b_i$ are the nonconstant coordinates is equal to $\{a_1 a_2\cdots a_j, a_1 a_2\cdots a_{j-1}, \ldots,a_1,\emptyset,b_1,b_1 b_2,\ldots, b_1 b_2\cdots b_i\}$. Note that $i$ or $j$ could be equal to 0.  We claim that $R(d,n)$ has property $V$.
		
		Suppose it does not, say $w$ and $x$ are bisequences in $B(d,n)$ with $a_1a_2\cdots a_l$ an initial-segment of $w$ all of whose elements are in $x$ but not an initial-segment (or not in the same order as in $w$) of either of the sequences in $x=\{(b_1,b_2,\ldots,b_j);(c_1,c_2,\ldots,c_i)\}$.  Then $\{\emptyset,a_1,a_1 a_2,\ldots,a_1 a_2 \cdots a_l\}$ is a path contained in the $d$-cube containing the path $\{b_1 b_2\cdots b_j,b_1 b_2\cdots b_{j-1},\ldots,b_1,\emptyset,c_1,c_1 c_2,\ldots,c_1 c_2\cdots c_i\}$, but is not a subpath, a contradiction.
		
		Similarly, reversing the procedure, a family of bisequences with Property $V$ and size $B(d,n)$ can be used to construct $B(d,n)$ sub-$d$-cubes containing $\emptyset$ with exact copies of $P_{d+1}$.  Hence $\Gmaxin(P_{d+1},d,n)=B(d,n)$, and dividing by $\binom{n}{d}$ gives the desired equality.
	\end{proof}
\end{prop}

Since $b(d,n)=\exlin(P_{d+1},d,n)$ is a non-increasing function of $n$, we can define $b(d)$ to be equal to $\ds\lim_{n\to\infty}\frac{b(d,n)}{\binom{n}{d}}$.  So $b(d)$ is the limit as $n$ goes to infinity of the maximum fraction of $d$-subsets of $n$ which can be the sets of elements of a family $R(d,n)$ of bisequences with Property $V$.  We have

\[
	\pilin(P_{d+1},d)=\lim_{n\to\infty}\frac{\exlin(P_{d+1},d,n)}{\binom{n}{d}}=\lim_{n\to\infty}\frac{b(d,n)}{\binom{n}{d}}=b(d)
\]
and we can find $\pilin(P_{d+1},d)$ by finding $b(d)$.

Clearly if $R(d,n)$ is a family of $d$-bisequences with Property $V$, then no symbol can appear as the first element of some sequence and not as the first element of another.  Furthermore, the following properties are easy to verify.

\begin{lem}\label{triangle-free3}
	Suppose $R(3,n)$ is a family of 3-bisequences with Property $V$.
	\begin{enumerate}[label=(\roman*)]
		\item If $\{bxy;\emptyset\}$ and $\{bxz;\emptyset\}$ are in $R(3,n)$, then $\{byz;\emptyset\}$ is not.
		\item If $\{bxz;\emptyset\}$ and $\{byz;\emptyset\}$ are in $R(3,n)$ then $\{bxy;\emptyset\}$ is not.
		\item If $\{bx;c\}$ and $\{bx;d\}$ are in $R(3,n)$, then $\{dx;c\}$ is not.
		\item If $\{bx;c\}$ and $\{dx;c\}$ are in $R(3,n)$, then $\{dx;b\}$ is not.
	\end{enumerate}

\begin{proof}
	Let $A=\{i\in[n] : i \textrm{ is the first element of some sequence in a 3-bisequence in }R(3,n)\}$
 	and let $W=[n]\setminus A$.  Let $a= \frac{\abs{A}}{n}$ and $w=\frac{\abs{W}}{n}=1-a$.
	
	(i) If the assumption in (i) holds, then $\{byz;\emptyset\}$ cannot be in $B(3,n)$ because it has ``$by$'' as an initial segment, and that is a subset of one of the sequences in $\{bxy;\emptyset\}$ but is not an initial segment, violating property $V$.
	
	Statements (ii), (iii), and (iv) are just as easy to verify.
\end{proof}
\end{lem}

	\begin{proof}[Proof of Proposition \ref{pilP_4,3}]
		By Proposition \ref{path-bisequence} it suffices to show that $b(3)\leq \frac{3}{8}$.  Let $R(3,n)$ be a family of 3-bisequences with Property $V$.  Let $A$ be the set of elements in $[n]$ which appear as the first element of some sequence in $R(3,n)$,  let $W=[n]\setminus A$, and let $a=\abs{A}$ and $w=\abs{W}$.  For each $e\in A$, let $G_e$ be the graph with vertex set $W$ and edge set $\{[x,y] : \{eyx;\emptyset\}\rm{\ or\ }\{exy;\emptyset\} \textrm{ or one of their reversals is in }R(3,n)\}$.  By statements (i) and (ii) in Lemma \ref{triangle-free3}, $G_e$ is a triangle free graph for each $e \in A$.  Hence by Turan's theorem, at most $\frac{w^2}{4}$ unordered pairs $(x,y)$ of element $x$ and $y$ in $W$ can appear as edges in $G_e$.  That means that the total number of 3-subsets of $[n]$ which can be the set of elements of a 3-bisequence in $B(3,n)$ with one element in $A$ and two in $W$ is at most $\frac{w^2}{4}\cdot a$.  Similarly, for each element $x$ in $W$ we let $G_x$ be the graph with vertex set $A$ and edge set $\{[b,c] : \{bx,c\}\textrm{ or }\{bc,x\}\textrm{ or either of their reversals is in }R(3,n)\}$.  By statements (iii) and (iv) in Lemma \ref{triangle-free3}, $G_x$ is triangle free, and an identical argument to the one for $G_e$ shows that the total number of 3-subsets of $[n]$ which can be the set of elements of a 3-bisequence in $B(3,n)$ with two elements in $A$ and one in $W$ is at most $w\cdot\frac{a^2}{4}$.  If $B(3,n)$ is the size of $R(3,n)$ then
	         \begin{align*}
	         	B(3,n) &\leq a\cdot \frac{w^2}{4}+ w\cdot \frac{a^2}{4}\\
		               &= \frac{aw}{4}\cdot n
	         \end{align*}
		This is maximized when $a=w=\frac{n}{2}$, so $B(3,n)\leq\frac{n^3}{16}$ and $b(3)=\lim_{n \to \infty}\frac{B(3,n)}{\binom{n}{3}}\leq \frac{3}{8}$.
	\end{proof}

Further, this also shows Theorem \ref{P4theorem} holds.

\begin{proof}[Proof of Theorem \ref{P4theorem}]
	By Proposition \ref{path_density} and Proposition \ref{pilP_4,3}
	\[
		\frac{3}{8}\leq\pi(P_4,3)\leq\pil(P_4,3)\leq\frac{3}{8}.
	\]
\end{proof}

So Conjecture \ref{pathconj} holds for $d=3$.  Lending credence to this conjecture is that Baber's flag algebra upper bound for $\pi(P_5,4)$ is .19200000058, while the conjecture with $d=4$ gives $\frac{24}{125} = .192$.  
%
%
%
%
%

\section*{Acknowledgement}
We thank Zongchen Chen for his insightful observations on $d$-cube density.

\bibliographystyle{amsplain}
\bibliography{perfectpath2}
\end{document}